\documentclass[12pt,a4paper]{amsart}
\allowdisplaybreaks[1]

\usepackage[top=30truemm,bottom=30truemm,left=25truemm,right=25truemm]{geometry}

\DeclareFontEncoding{OT2}{}{}
\DeclareFontSubstitution{OT2}{cmr}{m}{l}

\DeclareFontFamily{OT2}{cmr}{\hyphenchar\font45 }
\DeclareFontShape{OT2}{cmr}{m}{l}{%
<5><6><7><8><9>gen*wncyr%
<10><10.95><12><14.4><17.28><20.74><24.88>wncyr10}{}

\DeclareMathAlphabet{\mathcyr}{OT2}{cmr}{m}{l}
\DeclareMathAlphabet{\mathcyb}{OT2}{cmr}{b}{l}
\SetMathAlphabet{\mathcyr}{bold}{OT2}{cmr}{b}{l}

\usepackage{amstext}
\usepackage{amsthm}
\usepackage{amssymb}

\newtheorem{thm}{Theorem}[section]
\newtheorem{lem}[thm]{Lemma}

\newtheorem{conj}[thm]{Conjecture}
\theoremstyle{definition}
\newtheorem{defn}[thm]{Definition}

\theoremstyle{remark}
\newtheorem{rem}[thm]{Remark}

\newcommand{\sha}{\mathbin{\widetilde{\mathcyr{sh}}}}

\begin{document}

\title[Linear relations of Ohno sums of MZVs]{Linear relations of Ohno sums of multiple zeta values}

\author{Minoru Hirose}
\address[Minoru Hirose]{Faculty of Mathematics, Kyushu University
 744, Motooka, Nishi-ku, Fukuoka, 819-0395, Japan}
\email{m-hirose@math.kyushu-u.ac.jp}

\author{Hideki Murahara}
\address[Hideki Murahara]{Nakamura Gakuen University Graduate School,
 5-7-1, Befu, Jonan-ku, Fukuoka, 814-0198, Japan} 
\email{hmurahara@nakamura-u.ac.jp}

\author{Tomokazu Onozuka}
\address[Tomokazu Onozuka]{Multiple Zeta Research Center, Kyushu University 744, Motooka, Nishi-ku,
Fukuoka, 819-0395, Japan}
\email{t-onozuka@math.kyushu-u.ac.jp}

\author{Nobuo Sato}
\address[Nobuo Sato]{Faculty of Mathematics, Kyushu University
 744, Motooka, Nishi-ku, Fukuoka, 819-0395, Japan}
\email{n-sato@math.kyushu-u.ac.jp}


\begin{abstract}
 Ohno's relation is a well-known family of relations among multiple zeta values, which can naturally be regarded as a type of duality for a certain power series which we call an Ohno sum. 
 In this paper, we investigate $\mathbb{Q}$-linear relations among Ohno sums which are not contained in Ohno's relation. 
 We prove two new families of such relations, and pose several further conjectural families of such relations. 
\end{abstract}

\maketitle

\section{Introduction}
The multiple zeta values (MZVs) are defined by the convergent series
\begin{align*}
 \zeta(k_1,\dots, k_r)
 :=\sum_{1\le m_1<\cdots <m_r} \frac{1}{m_1^{k_1}\cdots m_r^{k_r}} \in \mathbb{R}, 
\end{align*} 
where $(k_1,\dots, k_r)$ is an admissible index (sequence of positive integers whose last component is greater than $1$). 
For an index $\boldsymbol{k}=(k_1,\dots,k_r)$, we call $\left|\boldsymbol{k}\right|=k_1+\cdots+k_r$ its weight and $r$ its depth. 
\begin{defn}[Dual index]
For an admissible index 
 \[
  \boldsymbol{k}=(\underbrace{1,\ldots,1}_{a_1-1},b_1+1,\dots,\underbrace{1,\ldots,1}_{a_l-1},b_l+1) \quad (a_p, b_q\ge1),
 \]
we define the dual index of $\boldsymbol{k}$ by 
 \[
  \boldsymbol{k}^\dagger :=(\underbrace{1,\ldots,1}_{b_l-1},a_l+1,\dots,\underbrace{1,\ldots,1}_{b_1-1},a_1+1).
 \]
\end{defn}
For two indices $\boldsymbol{k}$ and $\boldsymbol{e}$ of the same depth, we denote by $\boldsymbol{k} \oplus \boldsymbol{e}$ the index obtained by componentwise addition. 
\begin{thm}[Ohno's relation; Ohno \cite{Oho99}] \label{ohno}
 For an admissible index $\boldsymbol{k}$ and a non-negative integer $m$, we have
 \begin{align*}
  \sum_{\substack{ |\boldsymbol{e}|=m \\ \boldsymbol{e}\in\mathbb{Z}_{\ge0}^{r} }}
  \zeta (\boldsymbol{k}\oplus\boldsymbol{e}) 
  =\sum_{\substack{ |\boldsymbol{e}'|=m \\ \boldsymbol{e}'\in\mathbb{Z}_{\ge0}^{r'} }}
  \zeta (\boldsymbol{k}^\dagger\oplus\boldsymbol{e}'), 
 \end{align*}
 where $r$ and $r'$ are the depths of indices $\boldsymbol{k}$ and $\boldsymbol{k}^\dagger$, respectively.
\end{thm}
Hereafter, we omit such conditions as $\boldsymbol{e}\in\mathbb{Z}_{\ge0}^{r}$ in the summation of the same type as above if there is no risk of confusion. 
Motivated by Ohno's relation, we introduce Ohno sums as follows:
\begin{defn}[Ohno sum]
 For an admissible index $\boldsymbol{k}$ and a non-negative integer $m$, 
 we define $\mathcal{O}_m(\boldsymbol{k})$, $\mathcal{O}(\boldsymbol{k})$, respectively by
 \begin{align*}
  \mathcal{O}_m(\boldsymbol{k})
  &:=\sum_{|\boldsymbol{e}|=m}
  \zeta (\boldsymbol{k}\oplus\boldsymbol{e}) \in\mathbb{R}, \\
  \mathcal{O}(\boldsymbol{k})
  &:=\sum_{m=0}^\infty \mathcal{O}_m(\boldsymbol{k})X^m \in\mathbb{R}[[X]].
 \end{align*}
\end{defn}
By using $\mathcal{O}(\boldsymbol{k})$, we can rewrite Ohno's relation as 
\begin{align} \label{ohno_eq}
 \mathcal{O}(\boldsymbol{k})=\mathcal{O}(\boldsymbol{k}^\dagger).
\end{align}
One may naturally ask whether (\ref{ohno_eq}) exhausts the relations among Ohno sums or not. 
The answer is negative in general. 
In fact, in weights 6 and 7, we have new relations
\begin{align} 
 &\mathcal{O}(1,2,3)+\mathcal{O}(1,3,2)+\mathcal{O}(3,1,2)
 -\mathcal{O}(2,4)-3\mathcal{O}(3,3)=0, \label{eq1} \\
 &\mathcal{O}(1,2,4)+\mathcal{O}(1,4,2)+\mathcal{O}(4,1,2)
 -\mathcal{O}(2,5)-2 \mathcal{O}(3,4)-2 \mathcal{O}(4,3)=0, \label{eq2} \\
 &\mathcal{O}(2,3,2)+\mathcal{O}(1,4,2)+\mathcal{O}(1,3,3)
 -\mathcal{O}(3,1,3)-\mathcal{O}(2,2,3)-\mathcal{O}(2,1,4)=0. \label{eq3} 
\end{align}
In higher weights, we have further new relations among $\mathcal{O}(\boldsymbol{k})$'s (see Table 1).
%
\vspace{2ex}
\begin{table}[!h]
 \begin{center}
  \caption{The dimensions of $\mathbb{Q}$-linear relations among $\mathcal{O}(\boldsymbol{k})$'s}
  \begin{tabular}{|c|r|r|r|r|r|r|r|r|r|r|r|r|r|} 
   \hline
   Weight of $\boldsymbol{k}$ & 2& 3& 4& 5& 6& 7& 8& 9& 10& 11& 12 & 13 \\ 
   \hline
   Relations spanned by (\ref{ohno_eq}) & 0 & 1& 1& 4& 6& 16& 28& 64& 120 & 256 & 496 & 1024 \\  
   \hline
   All relations (conjectural) & 0 & 1& 1& 4& 7& 18& 35& 80& 162 & 352 & 723 & 1530 \\ 
   \hline
  \end{tabular}
 \end{center}
\end{table}

The aim of this paper is to investigate the ``missing'' relations among $\mathcal{O}(\boldsymbol{k})$'s, that is, the relations which do not come from (\ref{ohno_eq}). 
As particular cases of such relations, we prove the following two families of relations (Theorems \ref{main1} and \ref{main2}) which generalize (\ref{eq3}) and (\ref{eq1}), (\ref{eq2}), respectively.

Unfortunately, these families still do not exhaust all the missing relations. 
The complete set of relations which gives the deficit of the dimension is still to be understood. 
In the search of further missing relations, we have found several new conjectural families of such relations. 
These conjectures will be stated in Section 4.
\begin{thm}[Double Ohno relation] \label{main1}
 Let $d$ and $n_0,\dots,n_{2d}$ be non-negative integers, and $\boldsymbol{k}$ the index
 \[
  (\{2\}^{n_0},1,\{2\}^{n_1},3,\dots,\{2\}^{n_{2d-2}},1,\{2\}^{n_{2d-1}},3,\{2\}^{n_{2d}}).
 \]
 Then, we have
 \begin{align*}
  \sum_{|\boldsymbol{e}|=m}
  \mathcal{O}(\boldsymbol{k}\oplus\boldsymbol{e}) 
  =\sum_{|\boldsymbol{e}'|=m}
  \mathcal{O}(\boldsymbol{k}^\dagger\oplus\boldsymbol{e}') 
 \end{align*}
 for a non-negative integer $m$,
 or equivalently, 
 \[
  \sum_{|\boldsymbol{e}_1|=m_1} 
  \sum_{|\boldsymbol{e}_2|=m_2}
  \zeta (\boldsymbol{k}\oplus\boldsymbol{e}_1\oplus\boldsymbol{e}_2) 
  =\sum_{|\boldsymbol{e}'_1|=m_1} 
  \sum_{|\boldsymbol{e}'_2|=m_2}
  \zeta (\boldsymbol{k}^\dagger\oplus\boldsymbol{e}'_1\oplus\boldsymbol{e}'_2) 
 \]  
 for non-negative integers $m_1$ and $m_2$.
\end{thm}
\begin{rem}
 The relation (\ref{eq3}) is the case $\boldsymbol{k}=(1,3,2)$ and $m=1$ of Theorem \ref{main1}.
\end{rem}

We denote by $\boldsymbol{k}\sha\boldsymbol{l}$ the formal sum of all shuffles of the indices $\boldsymbol{k}$ and $\boldsymbol{l}$ e.g., 
\[
 (a)\sha (b,c)=(a,b,c)+(b,a,c)+(b,c,a). 
\]
For an admissible index $\boldsymbol{k}$ and an integer $s\ge2$, we put
 \[
  F(s;\boldsymbol{k})
  :=\mathcal{O}((s)\sha \boldsymbol{k})-\mathcal{O}((s)\sha \boldsymbol{k}^\dagger).
 \]
\begin{thm}  \label{main2}
 For integers $s,t\ge2$, we have
 \[
  F(s;(t+1))=F(t;(s+1)).
 \]
\end{thm}
\begin{rem}
 The relations (\ref{eq1}) and (\ref{eq2}) are the cases $(s,t)=(2,3)$ and $(s,t)=(2,4)$ of Theorem \ref{main2}, respectively.
\end{rem}

This paper is organized as follows:
In Section 2, we prove Theorem \ref{main1}, and in Section 3, we prove Theorem \ref{main2}. 
In Section 4, we give several conjectural relations among $\mathcal{O}(\boldsymbol{k})$'s.

\section{Proof of the double Ohno relation}
\subsection{Notation}
We recall Hoffman's algebraic setup with a slightly different convention (see Hoffman \cite{Hof97}).
Let $\mathfrak{H}:=\mathbb{Q} \left\langle x,y \right\rangle$ be the non-commutative polynomial ring in two indeterminates $x$ and $y$.
We define a $\mathbb{Q}$-linear map $\mathit{Z}\colon y\mathfrak{H}x \to \mathbb{R}$ 
by $\mathit{Z}( yx^{k_1-1}\cdots yx^{k_r-1}):= \zeta (k_1,\ldots, k_r)$. 
Let $\tau$ be the anti-automorphism of $\mathfrak{H}$ that interchanges $x$ and $y$. 
Put $\widehat{\frak{H}}:=\mathbb{Q}\left\langle\left\langle x,y\right\rangle\right\rangle$. 
We define the map $\sigma$ as an automorphism of $\widehat{\frak{H}}[[t]]$ satisfying $\sigma(x)=x$ and $\sigma(y)=y(1-xt)^{-1}$, and the $\mathbb{Q}$-linear map $\sigma_m\colon\frak{H} \to \frak{H}$ as the homogeneous degree $m$ components of $\sigma$. 
We define the $\mathbb{Q}$-linear map $T\colon \mathfrak{H} \to \mathfrak{H}$ 
by $T(1):=1$ and $T(u_1\cdots u_r):=u_r\cdots u_1$ for $u_i\in\{x,y\} \,(i=1,\ldots,r)$. 
We also define the $\mathbb{Q}$-linear operators $L_y, R_x$ on $\frak{H}$ by $L_y(w)=yw, R_x(w)=wx \, (w\in\frak{H})$.

We put $I(k_1,\ldots,k_r):=yx^{k_1-1}\cdots yx^{k_r-1}$.
Then we have
\[
 \sum_{|\boldsymbol{e}|=m} I(\boldsymbol{k}\oplus\boldsymbol{e})
 =\sigma_{m} (I(\boldsymbol{k})) 
\]  
and
\begin{align*}
 &I(\{2\}^{n_0},1,\{2\}^{n_1},3,\dots,\{2\}^{n_{2d-2}},1,\{2\}^{n_{2d-1}},3,\{2\}^{n_{2d}}) \\
 &=y(xy)^{n_0}(yx)^{n_1+1}(xy)^{n_2+1}(yx)^{n_3+1}\cdots (xy)^{n_{2d-2}+1}(yx)^{n_{2d-1}+1}(xy)^{n_{2d}}x.
\end{align*}
Thus, Theorem \ref{main1} is restated as follows:
\begin{thm} \label{main1b}
 For $w\in y\,\mathbb{Q}\langle xy,yx\rangle x$ and non-negative integers $m_1,m_2$, we have
 \[
  \sigma_{m_1}\sigma_{m_2} (w-\tau(w))\in \ker\mathit{Z}. 
 \]
\end{thm}

\subsection{Preliminary to the proof}
\begin{lem} \label{lem}
 For $w\in y\,\mathbb{Q}\langle xy,yx\rangle x$ and non-negative integers $m_1,m_2$, we have
 \[
  \sigma_{m_1} \tau \sigma_{m_2} \tau (w)=\tau \sigma_{m_2} \tau \sigma_{m_1} (w).
 \] 
\end{lem}
\begin{proof}
 The statement is equivalent to 
 \[
  \sigma^{(1)} \tau \sigma^{(2)} \tau (w)
  =\tau \sigma^{(2)} \tau \sigma^{(1)} (w) \quad (w\in y\,\mathbb{Q}\langle xy,yx\rangle x),
 \] 
 where $\sigma^{(1)}$, $\sigma^{(2)}$ are $\mathbb{Q}[[t_1,t_2]]$-automorphism of $\widetilde{\mathfrak{H}}:=\widehat{\mathfrak{H}}[[t_1,t_2]]$ defined by 
 $\sigma^{(i)}(x)=x$, $\sigma^{(i)}(y)=y(1-xt_i)^{-1}$ for $i\in\{1,2\}$.

 Let $\tau'$ be the automorphism on $\widetilde{\mathfrak{H}}$ that interchanges $x$ and $y$, and
 $S_u \,(u\in\{x,y\})$ be a map  on $\widetilde{\mathfrak{H}}u$ defined by $S_u(vu) =uv$ for $v\in\widetilde{\mathfrak{H}}$. 
 We can easily check tha $\tau=\tau' T=T\tau'$ on $\widetilde{\mathfrak{H}}$ 
 and $T\sigma^{(2)} T=S_y^{-1} \sigma^{(2)} S_y$ on $\widetilde{\mathfrak{H}}y$. 
 Thus we have
 \begin{align*}
  \sigma^{(1)} \tau \sigma^{(2)} \tau 
  &=\sigma^{(1)} \tau'T \sigma^{(2)} T\tau' \\
  &=\sigma^{(1)} \tau' S_y^{-1} \sigma^{(2)} S_y \tau' \\ 
  &=\sigma^{(1)} S_x^{-1} \tau' \sigma^{(2)} S_y \tau' \\   
  &=S_x^{-1} \sigma^{(1)} \tau' \sigma^{(2)} \tau' S_x 
 \end{align*}
 on $\widetilde{\mathfrak{H}}x$. 
 By a similar calculation, 
 \[
  \tau \sigma^{(2)} \tau \sigma^{(1)} 
  =S_x^{-1} \tau' \sigma^{(2)} \tau' \sigma^{(1)} S_x 
 \]
 on $\widetilde{\mathfrak{H}}x$. 
 Since $\sigma^{(1)} \tau' \sigma^{(2)} \tau'$ and $\tau' \sigma^{(2)} \tau' \sigma^{(1)}$ are ring homomorphisms, 
 and $S_x(w)\in\mathbb{Q}\langle xy,yx \rangle$, 
 it is sufficient to check $\sigma^{(1)} \tau' \sigma^{(2)} \tau'(v)=\tau' \sigma^{(2)} \tau' \sigma^{(1)}(v)$ 
 for $v\in\{xy,yx\}$.
 Moreover, since $\sigma^{(1)} \tau' \sigma^{(2)} \tau'(yx)=\tau' (\tau' \sigma^{(1)} \tau' \sigma^{(2)} (xy))$ and 
 $\tau' \sigma^{(2)} \tau' \sigma^{(1)} (yx)=\tau' (\sigma^{(2)} \tau' \sigma^{(1)} \tau' (xy))$, 
 we have only to check the case $v=xy$ from the symmetry of $t_1$ and $t_2$.
 By the direct calculation, we have
 \begin{align*}
  \sigma^{(1)} \tau' \sigma^{(2)} \tau' (xy)
  &=\sigma^{(1)} \tau' \sigma^{(2)} (yx) \\
  &=\sigma^{(1)} \tau' (y(1-xt_2)^{-1}x) \\
  &=\sigma^{(1)} (x(1-yt_2)^{-1}y) \\
  &=\sigma^{(1)} (xy(1-yt_2)^{-1}) \\
  &=xy(1-xt_1)^{-1}(1-y(1-xt_1)^{-1}t_2)^{-1} \\
  &=xy(1-xt_1)^{-1}((1-xt_1-yt_2)(1-xt_1)^{-1})^{-1} \\
  &=xy(1-xt_1-yt_2)^{-1}
 \end{align*}
 and
 \begin{align*}
  \tau' \sigma^{(2)} \tau' \sigma^{(1)} (xy)
  &=\tau' \sigma^{(2)} \tau' (xy(1-xt_1)^{-1}) \\
  &=\tau' \sigma^{(2)} (yx(1-yt_1)^{-1}) \\
  &=\tau' (y(1-xt_2)^{-1}x (1-y(1-xt_2)^{-1}t_1)^{-1}) \\
  &=\tau' (yx(1-xt_2)^{-1}( (1-xt_2-yt_1)(1-xt_2)^{-1} )^{-1}) \\
  &=\tau' (yx (1-xt_2-yt_1)^{-1}) \\
  &=xy (1-xt_1-yt_2)^{-1}.
 \end{align*}
 This finishes the proof.
\end{proof}

\subsection{Proof of Theorem \ref{main1}}
\begin{lem} \label{sakki}
 Let $V$ be a $\mathbb{Q}$-vector space and $\widetilde{\mathit{Z}}\colon y\mathfrak{H}x\rightarrow V$ a $\mathbb{Q}$-linear map. 
 \begin{enumerate}
  \renewcommand{\labelenumi}{(\roman{enumi}).}
  \item If $(\sigma_m-\tau\sigma_m\tau)(w) \in\ker\widetilde{\mathit{Z}}$ for all $m\in\mathbb{Z}_{\ge0}$ and $w\in y\mathfrak{H}x$, 
   then 
   \[
    (\sigma_{m_1}\sigma_{m_2}-\tau\sigma_{m_1}\sigma_{m_2}\tau)(w) \in\ker\widetilde{\mathit{Z}}
   \]
   for all $m_1,m_2\in\mathbb{Z}_{\ge0}$ and $w\in y\mathbb{Q}\langle xy,yx\rangle x$.
  \item If $(\sigma_m-\sigma_m\tau)(w) \in\ker\widetilde{\mathit{Z}}$ for all $m\in\mathbb{Z}_{\ge0}$ and $w\in y\mathfrak{H}x$, 
   then 
   \[
    (\sigma_{m_1}\sigma_{m_2}-\sigma_{m_1}\sigma_{m_2}\tau)(w) \in\ker\widetilde{\mathit{Z}}
   \]
   for all $m_1,m_2\in\mathbb{Z}_{\ge0}$ and $w\in y\mathbb{Q}\langle xy,yx\rangle x$.
 \end{enumerate}
 For $w\in y\,\mathbb{Q}\langle xy,yx\rangle x$ and non-negative integers $m_1,m_2$, we have
  \[
   \sigma_{m_1} \tau \sigma_{m_2} \tau (w)=\tau \sigma_{m_2} \tau \sigma_{m_1} (w).
  \] 
\end{lem}
\begin{proof}
 First, we prove (i). 
 By the assumption and Lemma \ref{lem}, we have
 \begin{align*}
  \widetilde{\mathit{Z}}(\tau\sigma_{m_1}\sigma_{m_2}\tau(w)) 
  &=\widetilde{\mathit{Z}}(\tau\sigma_{m_1}\tau\tau\sigma_{m_2}\tau(w)) \\
  &=\widetilde{\mathit{Z}}(\sigma_{m_1}\tau\sigma_{m_2}\tau(w)) \\
  &=\widetilde{\mathit{Z}}(\tau\sigma_{m_2}\tau\sigma_{m_1}(w)) \\
  &=\widetilde{\mathit{Z}}(\sigma_{m_2}\sigma_{m_1}(w)) \\
  &=\widetilde{\mathit{Z}}(\sigma_{m_1}\sigma_{m_2}(w)).
 \end{align*} 
 Thus we get (i). 
 
 Next, we prove (ii). 
 By putting $m=0$, we have $\widetilde{\mathit{Z}}(v)=\widetilde{\mathit{Z}}(\tau(v))$ for $v\in y\mathfrak{H}x$. 
 Thus (ii) is an immediate consequence of (i). 
\end{proof}
\begin{proof}[Proof of Theorem \ref{main1b}]
 By Theorem \ref{ohno}, $(\widetilde{\mathit{Z}},V)=(\mathit{Z},\mathbb{R})$ satisfies the assumption of Lemma \ref{sakki} (ii). 
 Thus the theorem holds. 
\end{proof}
\begin{rem}
A $q$-analog of multiple zeta values ($q$-MZVs) defined by
\begin{align*}
 \zeta_{q}(k_1,\dots,k_r)
 :=\sum_{1\le m_1<\cdots<m_r } \frac{q^{(k_1-1)m_1+\cdots +(k_r-1)m_r}}{{[m_1]}^{k_1}\cdots{[m_r]}^{k_r}} \in \mathbb{R}[[q]], 
\end{align*}
where $[m]$ denotes the $q$-integer $[m] =(1-q^m)/(1-q)$, 
satisfy Ohno's relation of exactly the same form as usual MZVs (see Bradley \cite{Bra05}). 
Thus they also satisfy the double Ohno relations by Lemma \ref{sakki}.
\end{rem}

\subsection{Double Ohno relations for finite multiple zeta values}
Kaneko and Zagier \cite{KZ19} introduced the finite multiple zeta values (FMZVs). 
Set $\mathcal{A}:=\prod_p\mathbb{F}_p/\bigoplus_p\mathbb{F}_p$, where $p$ runs over all primes.
For positive integers $k_1,\dots,k_r$, the FMZVs are defined by
\begin{align*}
 \zeta_\mathcal{A}(k_1,\dots,k_r)
 :=\Biggl(\sum_{1\le m_1<\dots<m_r<p}\frac{1}{m_1^{k_1}\dotsm m_r^{k_r}}\bmod p\Biggr)_p\in\mathcal{A}.
\end{align*}
In this section, we prove an analog of the double Ohno relation for FMZVs.
To state our main theorem, we need another duality due to Hoffman:
\begin{defn}[Hoffman's dual index]
For an index 
 \[
  \boldsymbol{k}=(\underbrace{1,\ldots,1}_{a_1-1},b_1+1,\dots,\underbrace{1,\ldots,1}_{a_{l-1}-1},b_{l-1}+1,\underbrace{1,\ldots,1}_{a_l-1},b_l) \quad (a_p, b_q\ge1),
 \]
we define Hoffman's dual index of $\boldsymbol{k}$ by 
 \[
  \boldsymbol{k}^\vee :=(a_1,\underbrace{1,\ldots,1}_{b_1-1},a_2+1,\underbrace{1,\ldots,1}_{b_2-1},\dots,a_l+1,\underbrace{1,\ldots,1}_{b_l-1}).
 \]
\end{defn}
\begin{thm}[Ohno-type relation; Oyama \cite{Oya15}] \label{ohnoF}
 For a non-empty index $\boldsymbol{k}$ and  a non-negative integer $m$, we have
 \[ 
  \sum_{\left|\boldsymbol{e}\right|=m}
  \zeta_\mathcal{F} (\boldsymbol{k}\oplus\boldsymbol{e})
  =\sum_{\left|\boldsymbol{e}\right|=m}
  \zeta_\mathcal{F} ((\boldsymbol{k}^\vee \oplus\boldsymbol{e})^\vee ).
 \]
\end{thm}
For an admissible index $\boldsymbol{k}=(k_1\dots,k_{r})$, we write $\boldsymbol{k}^-:=(k_1\dots,k_{r-1},k_r-1)$.
\begin{thm}[Double Ohno relation for FMZVs] 
 Let $d$ and $n_0,\dots,n_{2d}$ be non-negative integers, and $\boldsymbol{k}$ the index
 \[
  (\{2\}^{n_0},1,\{2\}^{n_1},3,\dots,\{2\}^{n_{2d-2}},1,\{2\}^{n_{2d-1}},3,\{2\}^{n_{2d}}).
 \]
 Then, we have
 \[
  \sum_{\left|\boldsymbol{e}_1\right|=m_1} 
  \sum_{\left|\boldsymbol{e}_2\right|=m_2}
  \zeta_{\mathcal{F}} (\boldsymbol{k}^-\oplus\boldsymbol{e}_1\oplus\boldsymbol{e}_2) 
  =\sum_{\left|\boldsymbol{e}_1\right|=m_1} 
  \sum_{\left|\boldsymbol{e}_2\right|=m_2}
  \zeta_{\mathcal{F}} (((\boldsymbol{k}^-)^\vee \oplus\boldsymbol{e}_1\oplus\boldsymbol{e}_2)^\vee) 
 \]  
 for non-negative integers $m_1$ and $m_2$.
\end{thm}
\begin{proof}
 We define a $\mathbb{Q}$-linear map $\mathit{Z}_{\mathcal{F}}\colon y\mathfrak{H}x\rightarrow\mathcal{A}$ 
 by $\mathit{Z}_{\mathcal{F}}( yx^{k_1-1}\cdots yx^{k_r-1}):= \zeta_{\mathcal{F}} ((k_1,\ldots, k_r)^-)$. 
 We first note that Theorem \ref{ohnoF} is equivalent to $(\sigma_m-\tau\sigma_m\tau)(w) \in\ker\mathit{Z}_\mathcal{F}$ for all $m\in\mathbb{Z}_{\ge0}$ and $w\in y\mathfrak{H}x$, 
 since 
 \begin{align*}
  \mathit{Z}_\mathcal{F} (\sigma_m(I(\boldsymbol{l})))
  &=\sum_{|\boldsymbol{e}|=m} \zeta_\mathcal{F} ((\boldsymbol{l}\oplus \boldsymbol{e})^-) \\
  &=\sum_{|\boldsymbol{e}|=m} \zeta_\mathcal{F} (\boldsymbol{l}^-\oplus \boldsymbol{e})
 \end{align*}
 and 
 \begin{align*}
  \mathit{Z}_\mathcal{F} (\tau \sigma_m \tau (I(\boldsymbol{l})))
  &=\sum_{|\boldsymbol{e}|=m} 
   \zeta_\mathcal{F} (((\boldsymbol{l}^\dagger \oplus \boldsymbol{e})^\dagger)^-) \\
  &=\sum_{|\boldsymbol{e}|=m} 
   \zeta_\mathcal{F} (((\boldsymbol{l}^-)^\vee \oplus \boldsymbol{e})^\vee)
 \end{align*}
 for an admissible index $\boldsymbol{l}$.
 Since $I(\boldsymbol{k}) \in y\mathbb{Q}\langle xy,yx\rangle x$, we find
 \[
  (\sigma_{m_1}\sigma_{m_2}-\tau\sigma_{m_1}\sigma_{m_2}\tau)(I(\boldsymbol{k})) \in\ker{\mathit{Z}_\mathcal{F}} 
 \] 
 by Lemma \ref{sakki} (i). 
 This completes the proof since $\mathit{Z}_\mathcal{F} (\sigma_{m_1}\sigma_{m_2} (I(\boldsymbol{k})))$
 and $\mathit{Z}_\mathcal{F} (\tau \sigma_{m_1}\sigma_{m_2} \tau (I(\boldsymbol{k})))$
 give the left-hand side and the right-hand side of the theorem, respectively, by a similar argument 
 as above. 
\end{proof}

\section{Proof of Theorem \ref{main2}}
In this section, we prove Theorem \ref{main2}. 
\begin{defn}
 For an admissible index $\boldsymbol{k}$, positive integers $s,t\ge2$, and a non-negative integer $m$, we define
 \begin{align*}
  F_m(s;\boldsymbol{k})
  &:=\mathcal{O}_m((s)\sha \boldsymbol{k})-\mathcal{O}_m((s)\sha \boldsymbol{k}^\dagger),\\
  D_m(s,t)
  &:=F_m(s;(t+1))-F_m(t;(s+1)). \\
 \end{align*}
\end{defn}
By definition, Theorem \ref{main2} is equivalent to $D_m(s,t)=0$, and we prove Theorem \ref{main2} in this form. 
In what follows, we use the special case of  the harmonic product formula  $\zeta(k)\zeta(\boldsymbol{l})=\zeta((k) \ast \boldsymbol{l})$, 
where
\begin{align*}
 (k) \ast \boldsymbol{l}
 =(k) \sha \boldsymbol{l} +\sum_{i=1}^r (l_1,\dots,l_{i-1},l_i+k,l_{i+1},\dots,l_r)
\end{align*}
is the harmonic product of $(k)$ and $\boldsymbol{l}=(l_1,\dots,l_r)$.  
\begin{lem} \label{Fm} 
 For integers $s\ge2$, $t\ge1$, and $m\ge0$, 
 \begin{align*} 
  F_m(s;(t+1))
  &=-(m+1)\zeta(s+t+1+m) \\
  &\quad +\sum_{m_1+m_2=m}\sum_{i=0}^{t-2}\mathcal{O}_{m_2}(\{1\}^i,s+m_1+1,\{1\}^{t-2-i},2)\\
  &\quad +\sum_{m_1+m_2=m}\mathcal{O}_{m_2}(\{1\}^{t-1},s+m_1+2).
 \end{align*}
\end{lem}
\begin{proof}
 By definition, we have
 \begin{align*} 
  F_m(s;(t+1))
  =\mathcal{O}_m((s)\sha (t+1))-\mathcal{O}_m((s)\sha (t+1)^\dagger).
 \end{align*}
 By the harmonic product formula, 
 \begin{align} \label{AAA}
 \begin{split}
  &\sum_{m_1+m_2=m}\mathcal{O}_{m_1}(s)\mathcal{O}_{m_2}(t+1)\\
  &=\sum_{m_1+m_2=m}\zeta((s+m_1)\ast(t+1+m_2))\\
  &=\sum_{m_1+m_2=m}\left\{\zeta((s+m_1)\sha(t+1+m_2))+\zeta(s+t+1+m_1+m_2) \right\} \\
  &=\mathcal{O}_m((s)\sha (t+1))+(m+1)\zeta(s+t+1+m).
 \end{split}
 \end{align}
 Since
 \begin{align*}
  (s+m_1)*((t+1)^\dagger\oplus \boldsymbol{e})
  &=(s+m_1)\sha((t+1)^\dagger\oplus \boldsymbol{e}) \\
  &\quad+\sum_{i=0}^{t-2}(\{1\}^i,s+m_1+1,\{1\}^{t-2-i},2)\oplus \boldsymbol{e}\\
  &\quad+(\{1\}^{t-1},s+m_1+2)\oplus \boldsymbol{e}
 \end{align*}
 for $m_1\ge0$ and $\boldsymbol{e}\in\mathbb{Z}_{\ge0}^t$, 
 again by the harmonic product formula, 
 we have
 \begin{align} \label{BBB}
 \begin{split}
  \sum_{m_1+m_2=m}\mathcal{O}_{m_1}(s)\mathcal{O}_{m_2}((t+1)^\dagger)
  &=\sum_{m_1+m_2=m}\sum_{|\boldsymbol{e}|=m_2}\zeta((s+m_1)*((t+1)^\dagger\oplus \boldsymbol{e}))\\
  &=\mathcal{O}_m((s)\sha (t+1)^\dagger)\\
  &\quad+\sum_{m_1+m_2=m}\sum_{i=0}^{t-2}\mathcal{O}_{m_2}(\{1\}^i,s+m_1+1,\{1\}^{t-2-i},2)\\
  &\quad+\sum_{m_1+m_2=m}\mathcal{O}_{m_2}(\{1\}^{t-1},s+m_1+2).
 \end{split}
 \end{align}
 Since $\mathcal{O}_{m_2}((t+1)^\dagger)=\mathcal{O}_{m_2}(t+1)$, we have
 \begin{align} \label{CCC}
  \sum_{m_1+m_2=m}\mathcal{O}_{m_1}(s)\mathcal{O}_{m_2}((t+1)^\dagger)
  =\sum_{m_1+m_2=m}\mathcal{O}_{m_1}(s)\mathcal{O}_{m_2}(t+1).
 \end{align}
 Combining (\ref{AAA}), (\ref{BBB}), and (\ref{CCC}), we obtain the lemma. 
\end{proof}
\begin{lem}
 For positive integers $s,t\ge3$ and $m\ge1$, 
 \begin{align*}
  D_{m-1}(s,t)=D_m(s-1,t)+D_m(s,t-1).
 \end{align*}
\end{lem}
\begin{proof}
 Since
 \begin{align*} 
  &\sum_{i=0}^{t-2} (\{1\}^i,s+m_1+1,\{1\}^{t-2-i},2)\\
  &=(s+m_1+1)\sha (t)^\dagger 
   -(\{1\}^{t-2},2,s+m_1+1), \\
 \end{align*}
 we have
 \begin{align*}
  F_m(s;(t+1))&=-(m+1)\zeta(s+t+1+m)\\
  &\quad+\sum_{m_1+m_2=m} \mathcal{O}_{m_2}((s+m_1+1)\sha (t)^\dagger)\\
  &\quad-\sum_{m_1+m_2=m} \mathcal{O}_{m_2}(\{1\}^{t-2},2,s+m_1+1)\\
  &\quad+\sum_{m_1+m_2=m} \mathcal{O}_{m_2}(\{1\}^{t-1},s+m_1+2)
 \end{align*}
 for $s,t\ge2$ and $m\ge0$ by Lemma \ref{Fm}. 
 Thus we find
 \begin{align*}
  &F_{m}(s-1;(t+1)) -F_{m-1}(s;(t+1))\\
  &=-\zeta(s+t+m)
   +\mathcal{O}_m((s)\sha (t)^\dagger)
   -\mathcal{O}_m(\{1\}^{t-2},2,s)
   +\mathcal{O}_m(\{1\}^{t-1},s+1)
 \end{align*}
 for $s\ge3$, $t\ge2$, and $m\ge1$.
 Therefore, 
 \begin{align*}
  &D_{m-1}(s,t)-D_m(s-1,t)-D_m(s,t-1) \\
  &=-(F_{m}(s-1;(t+1))-F_{m-1}(s;(t+1)))
   +(F_{m}(t-1;(s+1))-F_{m-1}(t;(s+1))) \\
  &\quad +F_{m}(t;(s))-F_{m}(s;(t)) \\
  &=\zeta(s+t+m)
   -\mathcal{O}_m((s)\sha (t)^\dagger)
   +\mathcal{O}_m(\{1\}^{t-2},2,s)
   -\mathcal{O}_m(\{1\}^{t-1},s+1)\\
  &\quad -\zeta(s+t+m)
   +\mathcal{O}_m((t)\sha(s)^\dagger)
   -\mathcal{O}_m(\{1\}^{s-2},2,t)
   +\mathcal{O}_m(\{1\}^{s-1},t+1) \\
  &\quad +F_{m}(t;(s))-F_{m}(s;(t))
 \end{align*}
 for $s,t\ge3$ and $m\ge1$.
 Since 
 $\mathcal{O}_m(\{1\}^{t-2},2,s)=\mathcal{O}_m(\{1\}^{s-2},2,t)$ and
 $\mathcal{O}_m(\{1\}^{t-1},s+1)=\mathcal{O}_m(\{1\}^{s-1},t+1)$
 by Ohno's relation, 
 and $F_{m}(t;(s))-F_{m}(s;(t))=-\mathcal{O}_m((t)\sha(s)^\dagger) +\mathcal{O}_m((s)\sha(t)^\dagger)$
 by the definition of $F_m$, 
 it readily follows that
 \[
  D_{m-1}(s,t)-D_m(s-1,t)-D_m(s,t-1)=0. \qedhere
 \] 
\end{proof}

\begin{proof}[Proof of Theorem \ref{main2}]
 Recall that Theorem \ref{main2} is equivalent to $D_m(s,t)=0$ for $s,t\ge2$ and $m\ge0$. 
 Since 
 \[
  D_n(s,t)
  =D_{n+1}(s-1,t) +D_{n+1}(s,t-1) 
 \]
 for $s,t\ge3$ and $n\ge0$
 by the previous lemma,
 we can reduce the theorem to $D_m(2,t)=0$ and $D_m(s,2)=0$ for $s,t\ge2$ and $m\ge0$. 

 By the symmetry $D_m(s,t)=-D_m(t,s)$, 
 it suffices to show $D_m(s,2)=0$, that is, $F_m(s;(3))-F_m(2;(s+1))=0$. 
 By Lemma \ref{Fm}, we have
 \begin{align*}
  F_m(s;(3))
  &=-(m+1)\zeta(s+m+3)\\
  &\quad+\sum_{m_1+m_2=m}\mathcal{O}_{m_2}(s+m_1+1,2)+\sum_{m_1+m_2=m}\mathcal{O}_{m_1}(1,s+m_2+2)
 \end{align*}
 on one hand. 
 On the other hand, we have
 \begin{align*}
  F_m(2;(s+1))&=\mathcal{O}_m((2)\sha (s+1))-\mathcal{O}_m((2)\sha (\{1\}^{s-1},2))\\
  &=( \mathcal{O}_m(2,s+1)+\mathcal{O}_m(s+1,2) )
   -\biggl( \sum_{\substack{x+y=s+3\\x,y\ge2}} \mathcal{O}_m(x,y)+\mathcal{O}_m(2,s+1) \biggr) \\
  &=-\sum_{\substack{x'+y'=s\\x',y'>0}}\mathcal{O}_m(x'+1,y'+2) \\
  &=-\sum_{\substack{m_1+m_2=m \\m_1,m_2\ge0 }}
   \sum_{\substack{ x'+y'=s \\x',y'>0 }} \zeta(x'+m_1+1,y'+m_2+2)
 \end{align*}
 by Ohno's relation.  
 Since
 \begin{align*}
  \mathcal{O}_{m_1} (1,s+m_2+2)
  &=\sum_{\substack{ a+b=s+m_1+m_2 \\ 0\le a\le m_1,0\le b }} \zeta(a+1,b+2), \\
  \mathcal{O}_{m_2} (s+m_1+1,2)
  &=\sum_{\substack{ a+b=s+m_1+m_2 \\ 0\le a, 0\le b\le m_2 }} \zeta(a+1,b+2), 
 \end{align*}
 and 
 \begin{align*}
  \sum_{\substack{ x'+y'=s \\x',y'>0 }} \zeta(x'+m_1+1,y'+m_2+2)
  =\sum_{\substack{ a+b=s+m_1+m_2 \\ a>m_1,b>m_2 }} \zeta(a+1,b+2), 
 \end{align*}
 we have
 \begin{align*}
  &\mathcal{O}_{m_1} (1,s+m_2+2)
  +\mathcal{O}_{m_2} (s+m_1+1,2) 
  +\sum_{\substack{ x'+y'=s \\x',y'>0 }} \zeta(x'+m_1+1,y'+m_2+2) \\
  &=\sum_{\substack{ a+b=s+m_1+m_2 \\ a\ge0,b\ge0 }} \zeta(a+1,b+2). 
 \end{align*}
 Therefore, 
 \begin{align*}
  &F_m(s;(3))-F_m(2;(s+1)) \\
  &=-(m+1)\zeta(s+m+3)
   +\sum_{\substack{m_1+m_2=m \\m_1,m_2\ge0 }} 
    \sum_{\substack{ a+b=s+m_1+m_2 \\ a\ge0,b\ge0 }} \zeta(a+1,b+2) \\
  &=-(m+1)\zeta(s+m+3)
   +(m+1) \sum_{\substack{ a+b=s+m \\ a\ge0,b\ge0 }} \zeta(a+1,b+2) \\
  &=0. \qedhere
 \end{align*}
\end{proof}

\section{Conjectures}
We conclude the paper by stating several families of conjectural relations among Ohno sums discovered by numerical computation. 
\begin{conj}
 For integers $s\ge2$ and $m,n\ge0$, 
 \begin{align*}
  &\mathcal{O}((s)\sha(\{2\}^{m},1,\{2\}^{n+1}))
   -\mathcal{O}((s)\sha(\{2\}^{n},3,\{2\}^{m}))\\
  &=\sum_{\substack{ a+b=s+3 \\ a\ge2, b\ge3}}
    \left(\sum_{\substack{ i+j=m \\ i,j\ge0 }}
    \mathcal{O}(\{2\}^{i},a,\{2\}^{n},b,\{2\}^{j})
   -\sum_{\substack{ i+j=m-1 \\ i,j\ge0 }}
    \mathcal{O}(\{2\}^{i},a,\{2\}^{n+1},b,\{2\}^{j})\right).
 \end{align*}
\end{conj}
\begin{conj}
 For integers $s\ge3$ and $n\ge0$, 
 \begin{align*}
  \mathcal{O}((s)\sha(\{3\}^{n}))
  &=\sum_{\substack{ i+j+k=n \\ i,j,k\ge0 \\ a_{0}+\cdots+a_{j}=s+3j \\ a_{0},\ldots,a_{j-1}\ge2, a_j\ge3 }}
   (-1)^{j} \mathcal{O}(\{1,2\}^{i},a_{0},\ldots,a_{j},\{1,2\}^{k})\\
  &\quad 
   +\sum_{\substack{ i+j+k=n-1 \\ i,j,k\ge0 \\ a_{0}+\cdots+a_{j}=s+3j \\ a_{0},\ldots,a_{j}\ge2 }}
   (-1)^{j} \mathcal{O}(\{1,2\}^{i},1,a_{0},\ldots,a_{j},2,\{1,2\}^{k}).
 \end{align*}
\end{conj}
\begin{conj}
 For an integer $s\ge2$,  
 \begin{align*}
  &\mathcal{O}((s)\sha(1,1,3))-\mathcal{O}((s)\sha(1,4))\\
  &=-\sum_{\substack{ a+b=s+3 \\ a\ge1,b\ge3 }}
    \mathcal{O}(a,b,2)
   +\sum_{\substack{ a+b=s+4 \\ a\ge2,b\ge3 }}
    \mathcal{O}(a,1,b)+\sum_{\substack{ a+b=s+2 \\ a\ge2,b\ge2 }}
    \mathcal{O}(a,b,3).
 \end{align*}
\end{conj}
\begin{conj}
 For an integer $s\ge2$,  
 \begin{align*}
  &\mathcal{O}((s)\sha(4,2))-\mathcal{O}((s)\sha(2,1,1,2))\\
  &=\sum_{\substack{ a+b=s+2 \\ a\ge1,b\ge1 }}
    \mathcal{O}(a,b,2,2)
   -\sum_{\substack{ a+b=s+2 \\ a\ge2,b\ge2 }}
    \mathcal{O}(a,2,b,2)\\
  &\quad -\sum_{\substack{ a+b=s+3 \\ a\ge2,b\ge2 }}
    \mathcal{O}(2,a,1,b)
   -\sum_{\substack{ a+b=s+3 \\ a\ge2,b\ge3 }}
    \mathcal{O}(a,b,1,2)
   +\sum_{\substack{ a+b=s+3 \\ a\ge3,b\ge2 }}
    \mathcal{O}(a,1,2,b)\\
   &\quad -\sum_{\substack{ a+b=s+4 \\ a\ge2,b\ge3 }}
    \mathcal{O}(a,b,2)+\mathcal{O}(3,2,s+1)-\mathcal{O}(s+1,2,3).
 \end{align*}
\end{conj}
\begin{conj}
 For integers $s,t\ge2$ and $m\ge0$, 
 \[
  F(s;(\{2\}^{m})\sha(t+1))
  =F(t;(\{2\}^{m})\sha(s+1)).
 \]
\end{conj}
For example, 
for $m=1$,  
\[
 F(s;(t+1,2))+F(s;(2,t+1))
 =F(t;(s+1,2))+F(t;(2,s+1)), 
\]
and the case $m=0$ gives Theorem \ref{main2}.

\section*{Acknowledgements}
This work was supported by JSPS KAKENHI Grant Numbers, JP18J00982, JP18K13392, JP19J00835, and JP19K14511.

\end{document}